\documentclass[reqno,onecolumn,letterpaper,11pt]{amsart}

\usepackage{amssymb,amsmath,latexsym,amsfonts,times,url}
\usepackage[usenames,dvipsnames]{pstricks}
\usepackage{bbm}
\usepackage{epsfig}

\definecolor{dgreen}{rgb}{0.00,0.49,0.00}
\definecolor{dblue}{rgb}{0,0.08,0.75}
\RequirePackage[dvips,colorlinks,hyperindex]{hyperref} 
\hypersetup{linktocpage=true,citecolor=dblue,linkcolor=dgreen}

\newtheorem{theorem}{Theorem}
\newtheorem{lemma}[theorem]{Lemma}
\newtheorem{proposition}[theorem]{Proposition}
\newtheorem{definition}[theorem]{Definition}

\newcommand{\R}{\mathbb{R}}   
\newcommand{\N}{\mathbb{N}}   
\renewcommand{\P}{\mathcal{P}} 
\renewcommand{\l}{\ell}  

\renewcommand{\P}{\mathcal{P}}
\newcommand{\AC}{AC_{\rm loc}(\R)}
\newcommand{\Lp}{L^p_{\rm loc}(\R)}
\newcommand{\Ll}{L^1_{\rm loc}(\R)}
\newcommand{\A}{{\mathcal A\hspace{0.2ex}}}
\newcommand{\F}{{\mathcal F}_0(\R)}
\DeclareMathOperator{\VI}{VI}

\title[Dynamic equilibria in fluid queueing networks]{Dynamic equilibria in fluid queueing networks}
\author[R. Cominetti]{Roberto Cominetti}
\author[J.~R. Correa]{Jos\'e R. Correa}
\author[O. Larr\'e]{Omar Larr\'e}
\dedicatory{Departamento de Ingenier\'{\i}a Industrial, Universidad de Chile\\ {\tt \url{{rccc,jcorrea,olarre}@dii.uchile.cl}}}
\date{January 2014}
\thanks{Supported by N\'ucleo Milenio Informaci\'on y Coordinaci\'on en Redes ICM/FIC P10-024F
and FONDECYT 1100046. An extended abstract of a preliminary version of this paper appeared as \cite{ccl}.}

\begin{document}

\maketitle
\begin{abstract}
This paper continues the study of equilibria for flows over time in the fluid queueing model recently 
considered by Koch and Skutella \cite{KochSk:2}. We provide a constructive proof for the existence and 
uniqueness of equilibria in the case of a single origin-destination with piecewise constant inflow rates, 
through a detailed analysis of the static flows obtained as derivatives of a dynamic equilibrium. 
We also give a nonconstructive existence proof of equilibria when the inflow rates belong to $L^p$ 
including the extension to multiple origin-destinations.
\end{abstract}

\section{Introduction}
Understanding time varying flows over networks is relevant in contexts where a steady
state is rarely observed such as urban traffic or the Internet. Frequently, these systems are characterized by a 
lack of coordination among the participating agents and have to be considered from a game theoretic perspective.

Research in flows over time was initially focused in optimization. The first to consider such questions
in a discrete time setting were Ford and Fulkerson \cite{FyF:2,FyF:1} who designed an algorithm
to compute a flow-over-time carrying the maximum possible flow from a source to a sink in a given timespan. Gale
\cite{Gale:1} then showed the existence of a flow pattern that achieves this optimum simultaneously 
for all time horizons. These results were extended to continuous time by Fleischer and Tardos
\cite{Tardos:1}, and Anderson and Philpott \cite{Anderson:1}, respectively. We refer
to Skutella \cite{Skutella:1} for an excellent survey.

The study of flows over time when flow particles act selfishly has mostly been considered in the 
transportation literature. The seminal paper  by Friesz, Bernstein, Smith, Tobin, and Wie \cite{fri} 
(see also the book \cite[Ran and Boyce]{RanBoyce}) proposed a general framework in the form of a 
variational inequality for which, unfortunately, little is known in terms of existence, uniqueness 
and characterization of solutions. Under suitable assumptions, an existence result  
was eventually obtained by Zhu and Marcotte \cite{Zhu}.
Also, Meunier and Wagner \cite{meu} establish the existence of dynamic equilibria using an 
alternative specification of the model and exploiting general results for games with a continuum of players. 
Recently Koch and Skutella \cite{KochSk:2} studied a more specific model, which can be traced back to
 Vickrey \cite{Vickrey}, in which there is an inflow stream at a single source that travels across 
the network towards a sink through edges that are characterized by a travel time or latency
and a {\em per-time-unit} capacity. The model is a fluid approximation of a queueing system and 
can be seen as a special case of Friesz {\em et al.}'s framework, though it does not satisfy the 
assumptions required for the existence result of Zhu and Marcotte. This fluid queueing model was recently 
considered by Bhaskar, Fleischer and Anshelevich \cite{BFA} to investigate the price-of-anarchy in Stackleberg 
routing games. A more detailed comparison with the related literature is postponed until \S\ref{related-work}.

\medskip\noindent{\em Our Contribution.}
This paper considers flows over time for the fluid queueing model as in Koch and Skutella \cite{KochSk:2}, 
and is an outgrowth of our preliminary work \cite{ccl,lar}. We provide a constructive (algorithmic) 
proof for the existence and uniqueness of equilibria, exploiting the key concept
of {\em thin flow with resetting} introduced by Koch and Skutella: a static flow together with an associated 
labeling that characterize the time derivatives of an equilibrium. We actually consider a slightly more 
restrictive definition by adding a normalization condition. Using a fixed point formulation we show
that normalized thin flows exist, and then we prove that the labeling is unique.
As a by-product, this yields an exponential-time algorithm to compute a normalized thin flow
and shows that this problem belongs to the complexity class PPAD, 
though we conjecture that it might be solvable in polynomial time.
By integrating these thin flows we deduce the existence of an equilibrium for the case 
of a piecewise constant inflow rate, and we show that the equilibrium is unique within a natural family
of flows over time. Finally, we give a non-constructive existence proof when the inflow rates belongs 
to the space of $p$-integrable functions $L^p$ with $1<p<\infty$, and we discuss how the result extends to multiple origin-destination pairs.

\medskip\noindent{\em Organization of the paper.} Section \S \ref{model} describes the fluid queueing
model for flows over time. Section \S \ref{section:TFR} 
characterizes the time derivatives of a dynamic equilibrium using the notion of normalized thin flows with 
resetting, and proves the existence and uniqueness of the latter. In
\S \ref{section:EU} we exploit the previous results to give a constructive proof for the existence of an equilibrium 
in the case of a piecewise constant inflow rate, and we discuss the uniqueness of this equilibrium. In \S \ref{section:Lp} we 
present a non-constructive existence result for more general inflow rates, including the case of multiple origin-destinations.
Finally, in \S \ref{related-work} we compare our findings with previous results in the literature and state some open questions. 
Appendix \S\ref{L1AC} at the end summarizes some technical facts used in the paper.

\section{A fluid queue model for dynamic routing games}
\label{model}

Throughout this paper we  consider a network  $\mathcal{N}=(G,\nu,\tau, s, t, u) $ consisting of a directed graph $G$ with node
set $V$ and edge set $E$, a vector  $(\nu_e)_{e \in E}$ of positive numbers representing 
queue service rates, a vector $(\tau_e)_{e \in E}$  of nonnegative numbers representing link
travel times,  a source $s \in V$, a sink $t \in V$, and an inflow rate function $u:\R\!\to\R_+$ taken
from the set $\F$ of non-negative and locally integrable functions 
which vanish on the negative axis, that is $u(\theta)=0$ for a.e. $\theta<0$. We denote 
$U(\theta)\!=\!\int_0^\theta\! u(\xi)d\xi$ the {\em cumulative inflow} so that $U\!\in\!\AC$, the space of locally 
absolutely continuous functions.
For the precise definition of these functional spaces and some of its basic properties we refer to \S\ref{L1AC}
and \cite[Chapter 3]{leo}.

A continuous stream of particles is injected at the source $s$ at a time-dependent rate $u(\theta)$, and 
flows through the network towards the sink $t$. Particles arriving to an edge $e$ join a queue with 
service rate $\nu_e$ and, after leaving the queue, travel along the edge to reach its head after $\tau_e$ time units.  
Each infinitesimal inflow particle is interpreted as a player that seeks to complete its journey in the least 
possible time, so that equilibrium occurs when each particle travels along an $s$-$t$ 
shortest path. The relevant edge costs for a particle entering 
the network at time $\theta$, must consider the 
queueing delays induced by other particles  along its path by the time at which each edge is reached. 
This introduces intricate spatial and temporal dependencies among the flows that enter the network at different 
times, possibly at future dates if overtaking occurs. 

The rest of this section makes these notions more precise. 
For simplicity, and without loss of generality, we assume  that there is at most one edge between 
any pair of nodes in $G$, that there are no loops, and that for each node $v\in V$ there is a path from $s$ to $v$.
An edge $e \in E$ from node $v$ to node $w$ is written $vw$, while the forward and backward stars
of a node $v\in V$ are denoted $\delta^+(v)$ and $\delta^-(v)$. We also suppose that the sum of latencies along 
any cycle is positive, namely $\sum_{e \in \mathcal{C}}\tau_e > 0$ for every cycle $\mathcal{C}$ in $G$.

\subsection{Flows over time} \label{subsection:model}
The model is stated in terms of the flow rates on every edge.
A {\em flow-over-time} is a pair $f\!=\!(f^+,f^-)$ of arrays of  functions 
$f^+_e,f^-_e\in\F$ for each $e \in E$, representing  the rate at which flow enters the tail of $e$ 
and the rate of flow leaving the head of $e$ respectively.
The {\em cumulative inflow}, {\em cumulative outflow}
and {\em queue size} are defined as the $\AC$ functions
$$
\begin{array}{ccl}
F^+_e(\theta)&=&\mbox{$\int_0^{\theta}f^+_e (\xi)\,d\xi$,}\\[1ex]
F^-_e(\theta)&=&\mbox{$\int_0^{\theta} f^-_e (\xi)\,d\xi$,}\\[1ex]
z_e(\theta)&=&F^+_e (\theta) - F^-_e (\theta+\tau_e).
\end{array}
$$
Note that the expression for the queue size $z_e(\theta)$ accounts for the time 
$\tau_e$ required to reach the head of the link after leaving the queue.
We say that $f$ is {\em feasible} if for almost all $\theta$ the following are satisfied
\begin{itemize}
\item {\em capacity constraints:} $f^-_e(\theta) \leq \nu_e$ for each $e \in E$,
\item {\em non-deficit constraints:} $z_e(\theta) \geq 0$ for each $e \in E$,
\item {\em flow conservation constraints:}
\begin{align}\label{eq:flow_cons}
\sum_{e \in \delta^+(v)} f^+_e (\theta) -\!\!\!\!\sum_{e \in \delta^-(v)} f^-_e (\theta) 
=\left\{\begin{array}{cl}
u(\theta)&\quad\mbox{ for $v=s$ }\\[1ex]
0&\quad\mbox{ for $v \in V \setminus \{ s,t \}$}.
\end{array}
\right.
\end{align}
\end{itemize}

\subsection{Queue dynamics.}
We assume that queues {\em operate at capacity}, that is to say, for almost all $\theta$ we have 
\begin{equation}\label{con:exp}
\mbox{  }f_e^-(\theta+\tau_e)=\left\{\begin{array}{cl}
\nu_e&\mbox{if }z_e(\theta)>0,\\
\min\{f_e^+(\theta),\nu_e\}&\mbox{otherwise.}
\end{array}\right.
\end{equation}
This condition can be equivalently stated in terms of the queue length dynamics
\begin{equation}\label{dam0}
z'_e(\theta)=\left\{
\begin{array}{ll}
~f_e^+(\theta)\!-\!\nu_e&\mbox{if }z_e(\theta)>0\\ 
{}[f_e^+(\theta)\!-\!\nu_e]_+&\mbox{otherwise,}
\end{array}
\right.
\end{equation}
whose unique solution is given by
(see {\em e.g.} \cite[section \S 1.3]{pra})
\begin{equation}\label{dam}
z_e(\theta)=\max_{u\in[0,\theta]}\int_u^\theta\!\![f_e^+(\xi)-\nu_e]d\xi.
\end{equation}
This formula shows that  the inflow $f_e^+$ completely determines the queue length $z_e$,
and then the outflow $f_e^-$ is also uniquely determined by \eqref{con:exp}. 

The {\em queueing delay} experienced by a particle entering $e$ at time $\theta$ 
before it starts traversing the edge is defined as 
\begin{align}
q_e(\theta)=\min \{ q \geq 0 :\int_\theta^{\theta+q}\!\!\!f_e^- (\xi+\tau_e)\,d\xi=z_e(\theta) \}. \label{def:q}
\end{align}
We denote $W_e^\theta=[\theta,\theta+q_e(\theta))$ the interval on which the particle
waits in the queue and $Q_e=\{\theta:z_e(\theta)>0\}$ the instants at which the queue is nonempty. 
We observe that for all $\xi\in W_e^\theta$ the queue remains nonempty since
$$z_e(\xi)=z_e(\theta)+\int_\theta^{\xi}\![f_e^+(\xi)\!-\!f_e^-(\xi\!+\!\tau_e)]\,d\xi\geq 
z_e(\theta)-\int_\theta^{\xi}\!\!f_e^-(\xi\!+\!\tau_e)\,d\xi>0$$
and therefore  $Q_e=\cup_\theta W_e^\theta$.

\begin{proposition}
 \label{proposition:condQequiv}
A queue operates at capacity if and only if the following three conditions hold simultaneously 
\begin{itemize}
    \item[(a)]  Capacity constraints: $f_e^-(\theta)\leq\nu_e$ for almost all $\theta$,
    \item[(b)] Non-deficit constraints: $z_e(\theta)\geq 0$ for all $\theta$,
    \item[({}c)] Queueing delay: $q_e(\theta)=z_e(\theta)/\nu_e$ for all $\theta$.
\end{itemize}
\end{proposition}

\begin{proof} Suppose the queue operates at capacity. From \eqref{con:exp} 
we clearly have (a) while \eqref{dam} implies (b).
To prove ({}c) we observe that $\int_\theta^{\theta+q}\!f_e^-(\xi\!+\!\tau_e)\,d\xi \leq \nu_e q$ 
from which it follows that $q_e(\theta)\geq z_e(\theta)/\nu_e$. 
On the other hand, since the queue remains nonempty on $W_e^\theta$
condition \eqref{con:exp} implies $f_e^-(\xi+\tau_e)=\nu_e$ a.e. $\xi\in W_e^\theta$ and then
$$\int_\theta^{\theta+z_e(\theta)/\nu_e}\!\!\!\!\!\!f_e^-(\xi+\tau_e)\,d\xi = z_e(\theta)$$ 
which yields $q_e(\theta)=z_e(\theta)/\nu_e$.

Conversely, suppose (a)-({}c). From ({}c) we get
$\int_\theta^{\theta+q_e(\theta)}[f_e^-(\xi+\tau_e)-\nu_e]\,d\xi=0$
so that (a) gives $f_e^-(\xi+\tau_e)=\nu_e$ for almost all $\xi\in W_e^\theta$, and 
Lemma \ref{sorgenfrey} implies that this equality holds a.e. on 
$\cup_{\theta}W_e^\theta=Q_e$ proving the first case of \eqref{con:exp}.
For the second case, (b) and Lemma \ref{lemma:sac}:(a)$\Rightarrow$({}c) give that
almost everywhere $z_e(\theta)=0$ implies $0=z_e'(\theta)=f_e^+(\theta)\!-\!f_e^-(\theta\!+\!\tau_e)$ 
and therefore $f_e^-(\theta\!+\!\tau_e)=\min\{f_e^+(\theta),\nu_e\}$.
\end{proof}

\subsection{Link travel times}
The time at which a particle exits from an edge $e$ can be computed as the sum of the entrance time $\theta$, 
plus queueing delay, plus latency, {\em i.e.}  
\begin{align}\label{etf}
T_e(\theta)=\mbox{$\theta+\frac{z_e(\theta)}{\nu_e}+\tau_e$.}
\end{align}
For notational convenience we omit the dependence of $T_e$ on the flow $f$.
Clearly $T_e\in\AC$ and using \eqref{dam0} we can compute its derivative almost everywhere as
\begin{equation}\label{der}
T'_e(\theta)=
\left\{\begin{array}{cl}
\mbox{$\frac{1}{\nu_e}$}f_e^+(\theta)&\mbox{if }z_e(\theta)>0,\\[0.6ex]
\max\{1,\mbox{$\frac{1}{\nu_e}$}f_e^+(\theta)\}&\mbox{otherwise.}
\end{array}\right.
\end{equation}
Hence $T'_e(\theta)\geq 0$ so that $T_e$ is non-decreasing and thus particles traversing $e$ 
respect FIFO without overtaking. 
Moreover, all the flow that enters
$e$ up to time $\theta$ exits by time $T_e(\theta)$. Indeed, since the queue is nonempty over the 
interval $W_e^\theta$, service at capacity 
implies $f_e^-(\xi+\tau_e)=\nu_e$ for almost all $\xi\in W_e^\theta$
and then
\begin{eqnarray}\nonumber
F_e^-(T_e(\theta))&=&\int_0^{\theta+\tau_e}\!\!\!\!\!\!f_e^-(\xi)\,d\xi+\!\!\int_{\theta+\tau_e}^{T_e(\theta)}\!\!\!\!\!\!\nu_e\, d\xi\\\nonumber
&=&F_e^-(\theta\!+\!\tau_e)+z_e(\theta)\\ \label{throughput}
&=&F_e^+(\theta).
\end{eqnarray}

\subsection{Dynamic shortest paths.}
A flow particle entering a path $P=(e_1,e_2,...,e_k)$ at time $\theta$ 
will reach the endpoint of the path at the time
\begin{align}\label{path_time}
    \l^P\!(\theta) & = T_{e_k}\circ\cdots\circ T_{e_1}(\theta),
\end{align}
Thus, denoting $\P_w$ the set of all $s$-$w$ paths in $G$,
the earliest time at which a particle starting from $s$ at time $\theta$ can reach $w$ is given by
\begin{align}  \label{eqL1P}
    \l_w(\theta) & = \min_{P \in \P_w} \l^P\!(\theta).
\end{align}

These functions correspond to shortest paths with edge costs that consider the queueing delays along the 
path at the appropriate times, taking into account the time it takes to reach every edge. We refer to them as 
{\em dynamic shortest paths}.  
Since the $T_e$'s are absolutely continuous and non-decreasing, the same holds for their 
compositions $\l^P$ and  therefore also for the $\l_w$'s (see \cite[Ch.3]{leo} or \S\ref{L1AC}). 
The monotonicity of $T_e$ together with the non-deficit constraints and the fact 
that the sum of latencies on any cycle is positive, imply that dynamic shortest paths 
do not contain cycles and therefore 
\eqref{eqL1P} can also be computed by solving
\begin{equation} \label{eqL} 
\l_w(\theta) = \left\{
\begin{array}{cl}
\theta & \mbox{ for $w=s$} \\  
\min\limits_{e=vw \in \delta^- (w)} T_e(\l_v(\theta))  & \mbox{ for $w \neq s$}.
\end{array}\right.
\end{equation}

The {\em $\theta$-shortest-path graph} is defined as the acyclic graph $G_\theta=(V,E'_\theta)$
containing all the shortest paths at time $\theta$.
An edge $e =vw$ is in $E'_\theta$ if and only if $T_e(\l_v(\theta))\leq\l_w(\theta)$, or equivalently  $T_e(\l_v(\theta))=\l_w(\theta)$,
 in which case it is said to be {\em active}.
Note that an inactive edge has $T_e(\l_v (\theta))>\l_w(\theta)$ so by continuity it remains inactive nearby.
We also denote $\Theta_e$  the set of all times $\theta$ at which $e$ is active. 
Note that $E'_\theta$ and $\Theta_e$  depend on the given flow-over-time $f$.

\subsection{Dynamic equilibrium} \label{subsection:eq_def}
A feasible $s$-$t$ flow-over-time can be interpreted as a dynamic equilibrium by looking at each infinitesimal 
inflow particle as a player that travels from the source to the sink along an $s$-$t$ path that yields the 
least possible travel time. The following definition makes this notion precise.

\begin{definition}[Dynamic equilibrium]\label{dyneq}
A feasible flow-over-time $f$ is a \textit{\em dynamic equilibrium} if for each $e=vw\in E$  
we have $ f^+_e (\xi)=0$  for almost all $\xi\in\l_v(\R\!\setminus\!\Theta_e)$.
\end{definition}

\noindent{\sc Remark.} In \cite{KochSk:2}, Koch and Skutella consider a slightly 
different notion, which we call {\em strong dynamic equilibrium}, requiring that
$e\not\in E'_\theta \Rightarrow f^+_e(\l_v(\theta))\!=\!0$ for each 
$e\!=\!vw\in E$ and almost all $\theta$. This condition implies dynamic 
equilibrium (since $\l_v$ is absolutely continuous and maps null sets into null sets),
and it is in fact strictly stronger as illustrated in the example below.
The point is that the composition 
$f^+_e(\l_v(\theta))$ does not allow to identify functions $f_e^+$ that coincide 
almost everywhere. Indeed, since $\l_v(\cdot)$ may be constant over a nontrivial 
interval, a simple modification of $f^+_e$ at a single point may spoil the {\em almost 
everywhere} condition with respect to $\theta$. 
Definition \ref{dyneq} avoids this difficulty. 

\vspace{1ex}
\noindent
{\sc Example.}
{\em Consider the simple network in Figure \ref{fig11} with inflow function
$$u(\theta)=\left\{\begin{array}{ll}
2&\mbox{ if }~ 0\leq \theta<1\\ 
0&\mbox{ if }~ 1\leq \theta\leq 2\\ 
1&\mbox{ if }~ 2<\theta.
\end{array}\right.$$

\begin{figure}[!htb]
\centering
\scalebox{1.2} 
{
\begin{pspicture}(0,-1)(6,1.3)
\pscircle[linewidth=0.022,dimen=outer](0.2,0.026855469){0.2}
\pscircle[linewidth=0.022,dimen=outer](6.3,0.026855469){0.2}
\pscircle[linewidth=0.022,dimen=outer](3.3,0.026855469){0.2}

\psline[linewidth=0.022cm,arrowsize=0.05291667cm
2.0,arrowlength=1.4,arrowinset=0.4]{->}(1.1,0.026855469)(3.1,0.026855469)

\psline[linewidth=0.022](0.4,0.22685547)(1.1,0.22685547)(1.1,-0.17314453)(0.4,-0.17314453)
\psline[linewidth=0.022](3.6,0.05)(4.1470447,0.5)(3.8915424,0.7817482)(3.3529556,0.33461943)
\psline[linewidth=0.022](3.6,-0.05)(4.1470447,-0.5)(3.8915424,-0.7817482)(3.3529556,-0.33461943)

\usefont{T1}{ptm}{m}{n}
\rput(0.75,0.45){\tiny $\nu_a\!\!=\!1$}
\usefont{T1}{ptm}{m}{n}
\rput(3.45,0.8){\tiny $\nu_b\!\!=\!\!1$}
\usefont{T1}{ptm}{m}{n}
\rput(3.45,-0.8){\tiny $\nu_c\!\!=\!\!1$}
\usefont{T1}{ptm}{m}{n}
\rput(4.85,-1.1){\tiny $\tau_c\!=\!2$}
\usefont{T1}{ptm}{m}{n}
\rput(2.1,0.23){\tiny $\tau_a\!=\!1$}
\usefont{T1}{ptm}{m}{n}
\rput(4.85,1.15){\tiny $\tau_b\!=\!1$}
\usefont{T1}{ptm}{m}{n}
\rput(6.31,0.04){\tiny $t$}
\usefont{T1}{ptm}{m}{n}
\rput(3.31,0.04){\tiny $r$}
\usefont{T1}{ptm}{m}{n}
\rput(0.2,0.04){\tiny $s$}
\psbezier[linewidth=0.022,arrowsize=0.05291667cm
2.5,arrowlength=2.0,arrowinset=0.4]{->}(4.02,-0.6131445)(4.66,-1.1531445)(5.34,-1.0731446)(6.26,-0.15314454)

\psbezier[linewidth=0.022,arrowsize=0.05291667cm
2.5,arrowlength=2.0,arrowinset=0.4]{->}(4.02,0.6468555)(4.66,1.1868554)(5.34,1.1068555)(6.26,0.18685547)
\end{pspicture}
}
\caption{\label{fig11}
Dynamic equilibrium in a simple network.}
\end{figure}
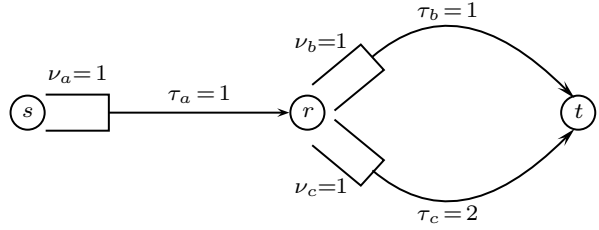

\noindent
The inflow of link $a$ is $f_a^+(\theta)=u(\theta)$ so that a queue builds up 
in the interval $[0,1]$ and is emptied during $[1,2]$ after which 
it stays empty with constant unit throughput. The outflow is $f_a^-(\theta)=\mathbbm{1}_{\{\theta\geq 1\}}$
so that links $b$ and $c$ can process all incoming flow without queueing. 
Since $\tau_c>\tau_b$ it follows that $c$ is never active and a dynamic equilibrium must 
send all the flow along the path $a$-$b$. A dynamic equilibrium is found by setting 
$f_c^+(\theta)\equiv f_c^-(\theta)\equiv 0$, $f_b^+(\theta)=\mathbbm{1}_{\{\theta\geq 1\}}$ and $f_b^-(\theta)=\mathbbm{1}_{\{\theta\geq 2\}}$,
with corresponding earliest-time functions $\l_s(\theta)=\theta$ and $\l_t(\theta)=\l_r(\theta)+1$ where
$$\l_r(\theta)=\left\{\begin{array}{cl}
1+\theta&\mbox{ for }~ \theta\in (-\infty,0]\cup[2,\infty)\\
1+2\theta&\mbox{ for }~ \theta\in[0,1]\\ 
3&\mbox{ for }~ \theta\in [1,2].
\end{array}\right.$$
This is in fact a strong dynamic equilibrium. However,
if we just modify $f_c^+$ by taking $f_c^+(3)>0$
we still have a dynamic equilibrium, but $f_c^+(\l_r(\theta))>0$ 
on the interval $[1,2]$ and strong equilibrium fails.
}

\vspace{2ex}
It is worth noting that at equilibrium all edges with positive queue must be active. 
Namely, let $E_\theta^*$ denote the set of links with positive queue
\begin{equation}
E^*_\theta=\{e=vw\in E: z_e(\l_v(\theta))>0\}.
\end{equation}
\begin{proposition}
\label{lemma:rightCont}
If $f$ is a dynamic equilibrium then $E^*_\theta\subseteq E'_\theta$ 
and we have
\begin{eqnarray}
E_\theta'&=&\{e=vw\in E:\l_w(\theta)\geq\l_v(\theta)+\tau_e\},\label{E'}\\
E_\theta^*&=&\{e=vw\in E:\l_w(\theta)>\l_v(\theta)+\tau_e\}.\label{E*}
\end{eqnarray}
\end{proposition}
\begin{proof} 
Let $e\!\in\! E^*_\theta$ and consider the largest $\theta'\!\leq\!\theta$ at which $e$ was active. 
Equilibrium implies $f_e^+(\xi)\!=\!0$ for almost all 
$\xi\!\in\! (\l_v(\theta'),\l_v(\theta)]$, so the queue must be nonempty throughout this interval and 
\eqref{der} gives $T_e'(\xi)\!=\!0$ almost everywhere. Hence $T_e$ is constant in this interval so that
$$T_e(\l_v(\theta))=T_e(\l_v(\theta'))=\l_w(\theta')\leq \l_w(\theta)$$ which yields 
 $e\in E_\theta'$ proving the inclusion $E^*_\theta\subseteq E'_\theta$.

To show \eqref{E'} we note that for $e\in E'_\theta$ we have 
$$\l_w(\theta)=T_e(\l_v(\theta))\geq\l_v(\theta)+\tau_e$$ where the inequality follows  from definition of $T_e$
and the non-deficit constraints.
Conversely, suppose that $\l_w(\theta)\geq\l_v(\theta)+\tau_e$. If $z_e(\l_v(\theta))=0$
this yields $\l_w(\theta)\geq T_e(\l_v(\theta))$ so that $e\in E_\theta'$, while in the case $z_e(\l_v(\theta))>0$ 
the same conclusion follows since we already proved that $E^*_\theta\subseteq E'_\theta$. 
 
A similar argument proves \eqref{E*}. For $e\in E_\theta^*$ we have 
$z_e(\l_v(\theta))>0$ and $e\in E_\theta'$, so that $\l_w(\theta)=T_e(\l_v(\theta))>\l_v(\theta)+\tau_e$.
Conversely, if $\l_v(\theta)+\tau_e<\l_w(\theta)$ the inequality $\l_w(\theta)\leq T_e(\l_v(\theta))$ 
and the definition of $T_e$ yield  $z_e(\l_v(\theta))>0$.
\end{proof}

Intuitively, at equilibrium all the flow routed through an edge $e=vw$ up to time
$\l_v(\theta)$ should reach $w$ before the optimal time $\l_w(\theta)$. This is in 
fact an equivalent characterization of dynamic equilibrium.
\begin{theorem}\label{char}
Let $f$ be a feasible $s$-$t$ flow-over-time. The following are equivalent
\begin{itemize}
\item[(a)] $f$ is a dynamic equilibrium,
\item[(b)] for each $e=vw$ and all $\theta$ we have $F_e^+(\l_v(\theta))=F_e^-(\l_w(\theta))$,
\item[({}c)]  for each $e=vw$  and almost all $\theta$ we have $e\not\in E'_\theta\Rightarrow f_e^+(\l_v(\theta))\l_v'(\theta)=0$.
\end{itemize}
\end{theorem}
\proof 
For each $\theta$ consider the interval $I_\theta=(\theta',\theta]$ with
 $\theta'\leq\theta$ the largest time such that $T_e(\l_v(\theta'))=\l_w(\theta)$
(it is well defined since $\l_w(\theta)\leq T_e(\l_v(\theta))$). 
We claim that $\Theta_e^c$ coincides with the union of the $I_\theta$'s.
Indeed,  for each $\theta\in \Theta^c_e$ we have $\theta'<\theta$ and therefore
 $\theta\in I_\theta$ so that $\Theta_e^c\subseteq\cup_\theta I_\theta$.
Conversely, for $\theta''\in I_\theta$ we have by definition of $\theta'$ that
$T_e(\l_v(\theta''))>\l_w(\theta)\geq \l_w(\theta'')$ 
so that $\theta''\in \Theta^c_e$ and then $\cup_\theta I_\theta\subseteq \Theta_e^c$. 

Now, invoking  \eqref{throughput}, for each $\theta$ we have
\begin{equation}\label{delta}
F_e^+(\l_v(\theta))-F_e^-(\l_w(\theta))=\int_{\l_v(\theta')}^{\l_v(\theta)}f_e^+(\xi)\,d\xi\geq 0,
\end{equation}
with equality iff $f_e^+$ vanishes almost everywhere on 
$(\l_v(\theta'),\l_v(\theta)]=\l_v(I_\theta)$. Lemma \ref{sorgenfrey} then shows that (b) holds 
iff $f_e^+(\xi)=0$ for almost all $\xi\in\cup_\theta \l_v(I_\theta)=\l_v(\Theta_e^c)$,
proving (b)$\Leftrightarrow$(a).
Similarly, a change of variables ({\em cf.} \S\ref{L1AC}) allows to rewrite  \eqref{delta} as
$$F_e^+(\l_v(\theta))-F_e^-(\l_w(\theta))=\int_{\theta'}^{\theta}\!f_e^+(\l_v(z))\l_v'(z)\,dz \geq 0,$$
with equality iff $f_e^+(\l_v(z))\l_v'(z)\!=\!0$ for almost all $z\in I_\theta$.
By Lemma \ref{sorgenfrey}, (b) holds iff this map vanishes almost everywhere
on $\cup_\theta I_\theta=\Theta_e^c$, proving (b)$\Leftrightarrow$({}c).
 \endproof

\begin{definition}[Cumulative flow]\label{flow}
The \textit{\em cumulative flow}
induced by a  dynamic equilibrium $f$ is defined as $x(\theta)\!=\!(x_e(\theta))_{e\in E}$  with
$x_e(\theta)\!=\!F_e^+(\l_v(\theta))\!=\!F_e^-(\l_w(\theta))$
for all $e=vw\in E$ and $\theta\in\R$.
\end{definition}
\noindent
Integrating the flow conservation constraints \eqref{eq:flow_cons} over
the interval $[0,\l_v(\theta)]$, it follows that for each $\theta\in\R$ the cumulative flow 
$x(\theta)$ is an $s$-$t$ flow of value $U(\theta)$,
\begin{align}\label{eq:flow_cons_int}
\sum_{e \in \delta^+(v)} x_e (\theta) -\!\!\!\!\sum_{e \in \delta^-(v)} x_e (\theta) 
=\left\{\begin{array}{cl}
U(\theta)&\quad\mbox{ for $v=s$}\\[1ex]
0&\quad\mbox{ for $v \in V \setminus \{ s,t \}$.}
\end{array}
\right.
\end{align}
Differentiating, for almost all $\theta$  we get that 
$x'(\theta)$ is an $s$-$t$ flow of value $u(\theta)$ with $x_e'(\theta)=0$ for $e\not\in E_\theta'$.

\subsection{Path formulation of dynamic equilibrium}
Since $G_\theta$ is acyclic, denoting  $\P$ the set of simple $s$-$t$ 
paths we may find a decomposition $u(\theta)=\sum_{P \in \P}h_P(\theta)$ into non-negative 
path-flows $h_P(\theta)\geq 0$ with 
$$x_e'(\theta)=\sum_{P\ni e}h_P(\theta).$$
Indeed, start with $y=x'(\theta)$ and consider the paths $P\in\P$ in a given 
order, setting $h_P(\theta)=\min_{e\in P}y_e$ and updating $y_e\leftarrow y_e-h_P(\theta)$
for $e\in P$. This yields a measurable decomposition $h_P\in \F$ such that $h_P(\theta)>0$ only for paths 
that belong to the $\theta$-shortest-path graph $G_\theta$.

It is appealing to take the latter as the definition of dynamic equilibrium. 
The difficulty is to properly define {\em shortest path} since this requires the exit-time 
functions $T_e$ which in turn require an appropriate flow-over-time $f$ to be associated with 
$h=(h_P)_{P\in\P}$. Since $f$ depends on how the flow $h$ propagates along the 
paths, both $f$ and $T_e$ must be determined simultaneously. This
{\em network loading} process typically requires additional conditions to be well defined,
such as an acyclic network structure or when link travel times are bounded away from zero
which is a natural and mild assumption
(see {\em e.g.} \cite{meu,XuF,Zhu}). Since we will not require network loading until 
\S \ref{section:Lp}, we defer its discussion to that section.

\section{Derivatives of dynamic equilibria: Normalized thin flows} \label{section:TFR}

The functions $x_e$ and $\l_w$ are absolutely continuous
and can be recovered by integrating their derivatives. In this section we characterize these 
derivatives, yielding a constructive method to find an equilibrium. Our characterization is 
closely related to the notion of {\em thin-flow with resetting} introduced by Koch and Skutella \cite{KochSk:2}. 

Recall that for almost all $\theta$ the derivative $x'(\theta)$ 
is an $s$-$t$ flow of size $u(\theta)$ with $x_e'(\theta)\!=\!0$ for $e\not\in E_\theta'$. 
On the other hand, clearly $\l_s'(\theta)=1$ while for $w\neq s$ we may combine \eqref{eqL}
and  \eqref{der} to get almost everywhere
$$\l_w'(\theta)=\min_{e=vw\in E_\theta'}T_e'(\l_v(\theta))\l_v'(\theta)
=\min_{e=vw\in E_\theta'}\rho_e(\l_v'(\theta),x_e'(\theta)),
$$
where for each $e=vw\in E_\theta'$ we set
$$\rho_e(\l_v',x_e')=\left\{\begin{array}{cl}
x_e'/\nu_e &\mbox{ if }e\in E_\theta^*\\
\max\{\l_v',x_e'/\nu_e\}&\mbox{ if }e\not\in E_\theta^*.
\end{array}\right.$$ 
Since $E_\theta'$ is acyclic, this allows to compute $\l_w'(\theta)$ by scanning the nodes $w$ in 
topological order. 

This motivates the next definition. Let $u_0\geq 0$ and $(E^*,E')$ a pair such that
$$\mbox{ $E^*\!\subseteq\! E'\!\subseteq\! E$ with $E'$ acyclic and for all $v\in V$ there is an $s$-$v$ path in $E'$.}\leqno (H)$$
We denote $K(E'\!,u_0)$ the nonempty, compact and convex 
set of all  static $s$-$t$ flows $x'=(x_e')_{e\in E}\geq 0$ of size $u_0$ with $x_e'=0$ for $e\not\in E'\!$. 
To each $x'\in K(E'\!,u_0)$ we associate the unique
labels given as above by $\l_s'=1$ and $\l_w'=\min_{e=vw\in E'}\rho_e(\l_v',x_e')$ for $w\neq s$.  Note that the
map $x'\mapsto\l'$ is continuous.

\begin{definition}[Normalized Thin Flow]
A flow $x'\!\in\! K(E'\!,u_0)$ is called a \textit{\em normalized thin flow (NTF) of value $u_0$ 
with resetting on $E^*\subseteq E'$} iff $x_e'=0$ for every edge $e=vw\in E'$ such that $\l_w'<\rho_e(\l_v',x_e')$.
\end{definition}

\begin{theorem}
\label{Th:K&S}
Let $f$ be a  dynamic equilibrium  and $\theta\in\R$ such that  
the right derivatives $u_0=\frac{dU}{d\theta^+}(\theta)$, $\l'_v=\frac{d\l_v}{d\theta^+}(\theta)$
and $x'_e=\frac{dx_e}{d\theta^+}(\theta)$ exist.
Then $x'$ is an NTF of value $u_0$ with resetting on $E^*_\theta\subseteq E_\theta'$,
with corresponding labels $\l'$.
\end{theorem}

\begin{proof} 
Differentiating \eqref{eq:flow_cons_int} it follows that 
$x'$ is an $s$-$t$ flow of value $u_0$.
Moreover, if $e\not\in E'_\theta$ then $e$ remains inactive on some interval $[\theta,\theta+\epsilon)$,
so the chain rule (see \S\ref{L1AC}) and equilibrium imply that on this interval
$x_e'(\xi)=f_e^+(\l_v(\xi))\l_v'(\xi)=0$  a.e., so $x_e(\cdot)$ is constant and $x_e'=0$. 
This proves that $x'\in K(E'_\theta,u_0)$. 

Let us show that $\l'$ are the corresponding labels. 
Clearly $\l'_s=1$.  For the rest of the argument we distinguish two more subsets of $E_\theta'$:
$E_{{}^+}^*$ contains the links which have a queue or are about to build one
with $z_e(\xi)>0$ on a small interval $(\theta,\theta+\epsilon)$, while $E_{{}^+}'$ includes the links
without queue at time $\theta$ but which are active along a strictly decreasing
sequence $\theta_n\downarrow\theta$. For $e\in E_{{}^+}'$ we have
$\l_w(\theta_n)=T_e(\l_v(\theta_n))\geq \l_v(\theta_n)+\tau_e$ and $\l_w(\theta)=\l_v(\theta)+\tau_e$
so that $\l_w(\theta_n)-\l_w(\theta)\geq\l_v(\theta_n)-\l_v(\theta)$ and
dividing by $\theta_n-\theta$ with $n\to\infty$ we get $\l_w'\geq\l_v'$.
Similarly, for $e\in E_\theta'\setminus E_{{}^+}^*$ we may take 
$\theta_n\downarrow\theta$ with $z_e(\l_v(\theta_n))=0$ so that $\l_w(\theta_n)\leq T_e(\l_v(\theta_n))=\l_v(\theta_n)+\tau_e$ 
and we get $\l_w'\leq\l_v'$.
Also, for $e=vw\in E'_\theta$ the capacity constraint gives for $\theta'\geq\theta$
$$x_e(\theta')- x_e(\theta)=\int_{\l_w (\theta)}^{\l_w (\theta')} \!\!\!\!\!f^-_e(\xi)\;d\xi\leq\nu_e(\l_w (\theta')-\l_w (\theta))$$
which implies $\l_w'\geq x'_e/\nu_e$, with  equality if $e\in E_{{}^+}^*$.
In summary
$$\begin{array}{cll}
\mbox{(a)}&\l_w'\geq\l_v'&\mbox{ for $e=vw\in E_{{}^+}'$}\\[0.5ex]
\mbox{(b)}&\l_w'\leq\l_v'&\mbox{ for $e=vw\in E_\theta'\setminus E_{{}^+}^*$}\\[0.5ex]
\mbox{({}c)}&\l_w'\geq x_e'/\nu_e&\mbox{ for $e=vw\in E_\theta'$}\\[0.5ex]
\mbox{(d)}&\l_w'=x_e'/\nu_e&\mbox{ for $e=vw\in E_{{}^+}^*$}.
\end{array}
$$
Combining (b) and (d) we get $\l_w'\leq\min_{e=vw\in E_\theta'}\rho_e(\l_v',x_e')$
with equality if there is some $e=vw\in E_\theta^*$. To prove the equality 
when no edge from $E_\theta^*$ is incident on $w$, choose any $\theta_n\downarrow\theta$ and 
a sequence of active edges $e_n\in E'_{\theta_n}$, and take a subsequence with $e_n=vw$ constant 
so that $e=vw\in E_{{}^+}'$. Then (a) and ({}c) combined give $\l_w'\geq\rho_e(\l_v',x_e')$. 
Altogether this proves $\l_w'=\min_{e=vw\in E_\theta'}\rho_e(\l_v',x_e')$ for $w\neq s$.

Let us finally show that $x'$ is an NTF. Suppose $x_e'>0$ on some $e=vw\in E_\theta'$ with $\l_w'<\rho_e(\l_v',x_e')$. The latter and (d) imply $e\not\in E_\theta^*$, while $x_e'>0$ gives $x_e(\theta')\!>\!x_e(\theta)$ for all $\theta'\!>\!\theta$ so $e$  
must be active on a sequence $\theta_n\!\downarrow\!\theta$ and $e\!\in\! E_{{}^+}'$.
Then (a) and ({}c) yield the contradiction $\l_w'\geq\rho_e(\l_v',x_e')$.
 \end{proof}

Theorem \ref{Th:K&S} derives the existence of NTF's  from a  dynamic equilibrium. 
To proceed in the other direction we study the existence of NTF's, and then 
by integration we reconstruct a dynamic equilibrium. 
\begin{theorem}
\label{teo:existTF}
Let $u_0\!\geq\! 0$ and $(E^*\!,E')$ satisfying $(H)$. Then there is an NTF of value $u_0$ with resetting  on $E^*\subseteq E'$.
\end{theorem}

\begin{proof} Let $K=K(E'\!,u_0)$ and observe that the NTF's are precisely the fixed-points of the set-valued map $\Gamma:K\to 2^K$ with nonempty convex compact values given by
$$\Gamma(x')=\{y'\in K: y'_e=0 \mbox{ for all $e\in E'$ such that $\l'_w<\rho_e(\l_v',x_e')$}\}$$
with $\l'$ the labels corresponding to $x'$ and $E^*$. Since $x'\mapsto\l'$ is continuous it follows 
that $\Gamma$ is upper-semi-continuous and the existence
of a fixed point $x'\in\Gamma(x')$ is guaranteed by Kakutani's Fixed Point Theorem.
\end{proof}

This result shows that finding an NTF belongs to the complexity class PPAD.
It also suggests a finite (exponential time) algorithm to compute an NTF: we guess the
set $E'_0$ of links $e\in E'$ that satisfy $\l'_w=\rho_e(\l_v',x_e')$, and then solve
$$\max_{(x',\l')}\mbox{$\left\{\sum_{w\in V}\l_w':x'\in K(E'_0,u_0);\l_s'=1;\l_w'\leq\min_{e=vw\in E'}\rho_e(\l_v',x'_e)\right\}$}.$$
The latter can be restated as a mixed integer linear program and solved in finite time. 
By considering all possible subsets $E'_0\subseteq E'$ the method eventually finds an NTF.
 
In general there may exist different NTF's, each one with its corresponding labels. We
show next that the labels in all of them coincide.
\begin{theorem}
\label{teo:uniTF}
Let $u_0\!\geq\! 0$ and $(E^*\!,E')$ satisfying $(H)$. Then the labels $\l'$ are the same for all NTF's 
of value $u_0$ with resetting on $E^*\subseteq E'$.
\end{theorem}
\begin{proof}
Let $x'$ and $y'$ be two NTF's with different labels $\l'\neq h'$, and 
suppose without loss of generality that  $S=\{v \in V : \l'_v > h'_v \}$ is nonempty. Consider the net flow across 
the boundary of $S$: since $x'$ and $y'$ satisfy flow conservation, setting $b_s=u_0$, $b_t=-u_0$ 
and $b_v=0$ for $v\in V\setminus\{s,t\}$, we get 
\begin{equation}\label{bdry}
x'(\delta^+(S))-x'(\delta^-(S))=\mbox{$\sum_{v\in S}$}b_v=y'(\delta^+(S))-y'(\delta^-(S)).
\end{equation}
For $e=vw \in \delta^+(S)$ we have $x'_e \leq y'_e$ since otherwise $x_e'>y_e'$ implies $x_e'>0$ and 
$\l_w'=\rho_e(\l_v',x_e')>\rho_e(h_v',y_e')\geq h_w'$ contradicting $w\not\in S$.
Similarly, $x'_e \geq y'_e$ for all $e=vw \in \delta^-(S)$  since $y_e'>x_e'$ implies $y_e'>0$ and  $h_w'=\rho_e(h_v',y_e')\geq \rho_e(\l_v',x_e')\geq\l_w'$ contradicting $w\in S$.
These inequalities and \eqref{bdry} imply $x_e'=y_e'$ for all $e\in \delta(S)$, with
$y'_e=0$ for $e\in\delta^-(S)$ since $y_e'>0$ yields a contradiction as before. 
Since $E'$ is acyclic, we may find $w\in S$ with all edges $e=vw\in E'$ belonging to $\delta^-(S)$.
Now, $\l_w'>h_w'\geq 0$ and $x_e'=0$ implies that $e \notin E^*$ for 
all these edges, and then $\rho_e(\l_v',x_e')=\l_v'$ as well as $\rho_e(h_v',y_e')=h_v'$, from which we
get the contradiction $h'_w=\min_{vw \in E' } h'_v\geq\min_{vw \in E' }  \l'_v=\l'_w$. 
\end{proof}

\section{Existence and uniqueness of dynamic equilibria}\label{section:EU}

Koch and Skutella \cite{KochSk:2} describe a method to extend an equilibrium for the case of a constant inflow rate
$u(\theta)\equiv u_0$. Given a feasible flow-over-time $f$ which satisfies the equilibrium conditions in $[0,\theta_k]$, the equilibrium is extended as follows:\\[1ex]
(1) Find $x'$ an NTF of value $u_0$ with resetting on $E^*_{\theta_k}\!\subseteq\! E'_{\theta_k}$, and
let $\l'$ denote the corresponding labels.\\[1ex]
(2) Compute $\theta_{k+1}=\theta_k+\alpha$ with  $\alpha > 0$ the largest value with
\begin{align}
 \l_w(\theta_k)+\alpha \l'_w-\l_v(\theta_k)-\alpha \l'_v  & \leq  \tau_e & &  \mbox{ for all } e=vw \not\in E'_{\theta_k} \quad\quad\quad\label{c1} \\
\l_w(\theta_k)+\alpha \l'_w-\l_v(\theta_k)-\alpha \l'_v  & \geq \tau_e & & \mbox{ for all } e=vw \in E^*_{\theta_k}\label{c2}
\end{align}
(3) Extend the earliest-time functions and the flow-over-time as\label{alg:pasoAct}
$$\begin{array}{ll}
\quad\l_v (\theta) = \l_v(\theta_k)+(\theta-\theta_k)\l'_v  & \text{ for } v \in V \text{ and }\theta \in [\theta_k,\theta_{k+1}]  \\[1ex]
\quad f_e^+(\xi) = x'_e/\l'_v &  \text{ for }  e=vw \in E \text{ and } \xi\in[\l_v(\theta_k),\l_v(\theta_{k+1}))\\[1ex]
\quad f_e^-(\xi) = x'_e/\l'_w &  \text{ for }  e=vw \in E \text{ and } \xi\in[\l_w(\theta_k),\l_w(\theta_{k+1}))\\[1ex]
\end{array}
$$
Theorems \ref{teo:existTF}  and \ref{teo:uniTF}  imply that $x'$ in step (1) exists and $\l'$ is unique. 
Moreover there are finitely many $\l'$, each one corresponding to a different pair $(E^*_\theta,E'_\theta)$.
The $\alpha$ computed in (2) is strictly positive so that each iteration extends the earliest-time functions
to a strictly larger interval. The conditions \eqref{c1} and \eqref{c2} correspond respectively to the maximum ranges on which 
the inactive edges remain inactive, and the positive queues remain positive. Hence,
for $\theta\in[\theta_k,\theta_{k+1})$ the pair $(E_\theta^*\!,E'_\theta)$ remains constant while 
at $\theta_{k+1}$ this pair changes and we must recompute the NTF. Note that when $\l_v'=0$
the update of $f_e^+$ does not extend its domain of definition, and similarly for
$f_e^-$ when $\l_w'=0$. As shown in \cite{KochSk:2} the extension maintains at all times 
the conditions for dynamic equilibrium in the strong sense (see Remark after Definition \ref{dyneq}).

This extension procedure can be used to establish the existence of a dynamic equilibrium.
Starting from the interval $(-\infty,\theta_0]$ with $\theta_0=0$ and zero flows, the extension can be iterated as long as required to find a new 
interval $[\theta_k,\theta_{k+1}]$ with $\theta_{k+1}>\theta_k$ at every step $k$. Eventually, $\theta_k$ may 
have a finite limit $\theta_\infty$: in this case, since the label functions are non-decreasing and have bounded derivatives, 
we can define the equilibrium at $\theta_\infty$ as the limit point of the label functions $\l$, and restart the extension 
process. A standard argument using Zorn's Lemma shows that a maximal solution is defined over all $\R_+$.  
Note that the $f$ constructed above is right-constant. 
\begin{definition} \label{rcl} A function $g:\R\to\R$ is called {\em right-constant} if for each $\theta\in\R$ there 
is an $\epsilon>0$ such that $g$ is constant on $[\theta,\theta+\epsilon)$. Similarly, $g$ is
{\em right-linear} if  for each $\theta$  it is affine on a small interval $[\theta,\theta+\epsilon)$.
\end{definition}
The extension method works even if the inflow rate function is piecewise constant, so we have the 
following existence result.

\begin{theorem}\label{Teo-existence}
Suppose that the inflow $u$ is piecewise constant, i.e,  there is an increasing sequence $(\xi_k)_{k \in \N}$ with $\xi_0=0$
such that $u(\cdot)$ is constant on each interval  $[\xi_k,\xi_{k+1})$. Then, there exists a  strong
dynamic equilibrium $f$ which is right-constant and whose label functions $\l$ are right-linear.
\end{theorem}

Dynamic equilibria are not unique in general. For instance, consider the network in Figure \ref{fig11}
but with $\tau_c=\tau_b$ so that any splitting of the outflow $f_a^-(\theta)=\mathbbm{1}_{\{\theta\geq 1\}}$ among 
these two links yields a dynamic equilibrium.  
Nevertheless, using Theorem \ref{teo:uniTF} one can prove
that the earliest-time functions in all sufficiently regular dynamic equilibria are the 
same and coincide with those given by the constructive procedure.

\begin{theorem}\label{Teo-uniqueness}
Suppose that  the inflow $u$ is piecewise constant.
Then, the earliest-time functions $(\l_v)_{v\in V}$ are the same for all 
dynamic equilibria $f$ which are right-continuous.
\end{theorem}
\proof When $f$ is right continuous it follows that the queue lengths $z_e(\theta)$, the exit-time functions
$T_e(\cdot)$ and the earliest-time functions $\l_v(\theta)$ are right-differentiable everywhere
with right-continuous derivatives. Theorem \ref{Th:K&S} implies that $\frac{d\l_v}{d\theta^+}(\cdot)$ are 
an NTF and Theorem \ref{teo:uniTF} shows that these derivatives are unique. Since they can take only 
finitely many values, continuity from the right imply that $\frac{d\l_v}{d\theta^+}(\cdot)$ is right-constant and 
$\l_v(\cdot)$ is right-linear. It follows that any two right-continuous dynamic equilibria must have the same 
earliest-time functions. Indeed, if these functions coincide up to time $\theta$, their right derivatives
at $\theta$ coincide and since they are right-linear they will also coincide on a nontrivial interval 
$[\theta,\theta+\epsilon]$. This implies that in fact the functions must coincide throughout $\R$.
\endproof


\section{Existence of equilibria for inflow rates in $L^p$}
\label{section:Lp}
The previous sections studied dynamic equilibria for a single origin-destination
with piecewise constant inflows. We consider next more general inflow rates and 
then extend the results to multiple origin-destinations. 
We proceed as in  \cite{fri} using a variational inequality for 
a path-flow formulation of  dynamic equilibrium. The analysis is non-constructive and
exploits the following particular case of the existence result \cite[Theorem 24]{Brezis}. 
Let  $(X,\|\cdot\|)$ be a reflexive Banach space and
 $\langle \cdot, \cdot\rangle$ the canonical pairing between $X$ and its dual $X^*$.
If $\A\!:\! K \to X^*$ is a weak-strong continuous map defined on   a nonempty, closed, bounded and convex 
subset $K\subseteq X$, then the following variational inequality  has a solution
$$
\mbox{ Find $x\in K$ such that $\langle \A x,y-x\rangle \geq 0$ for all $y \in K$.} \leqno{\VI(K,\A)}
$$

\subsection{Variational inequality formulation}
Let us consider first the case of a single origin-destination pair $st$ and 
an inflow rate $u\in L^p(0,T)$ where $T$ is a finite horizon and $0<p<1$.
We extend $u(\theta)\equiv0$ outside $[0,T]$ so that $u$ may be 
seen as a function in $\F$.
As before, let $\P$ be the set of paths connecting $s$ to $t$ and denote by $K$ the 
nonempty, bounded, closed and convex set of {\em feasible path-flows} given by
\begin{equation}\label{K}
K=\{h \in L^p(0,T)^{\P}: \mbox{$ \sum_{P \in \P} h_P=u$}\text{ and } h_P \geq 0\text { for all }P\in\P\}.
\end{equation}
The space $X\!=\!L^p(0,T)^\P$ is reflexive with dual $X^*\!=\!L^q(0,T)^{\P}$ where 
$\frac{1}{p}\!+\!\frac{1}{q}\!=\!1$. 
We will show that a dynamic equilibrium can be obtained by solving the variational inequality 
VI$(K,\A)$ with $\A\!:\!K\subseteq X \to X^*$ such that $\A_P(h)\in L^q(0,T)$ 
is the continuous function $\theta\mapsto\l^P\!(\theta)-\theta$ giving the time required to 
travel from $s$ to $t$ using path $P$. In order to properly define this map, our first task is to show that every $h\in K$ 
determines a unique feasible flow-over-time $f$, which in turn induces link travel times 
$T_e$ and path travel times $\l^P\!\!$. This is achieved by the network loading procedure 
described in the next subsection.
In \S\ref{contimes} we establish the weak-strong continuity of $\A$ and then in \S\ref{exi} we conclude
 the existence of a dynamic equilibrium. Finally, \S\ref{multiOD} extends the existence 
 result to multiple origin-destinations.

\subsection{Network loading}\label{netloading}
The following network loading procedure requires $\tau_e>0$ on every link $e$, 
 which we assume from now on.
Let $h=(h_P)_{P\in\P}$ be a given family of path-flows  with $h_P\in\F$ for all $P\in\P$.
A {\em network loading} is a flow-over-time $f=(f^+,f^-)$ together with non-negative and measurable
link-path decompositions
\begin{equation}\label{decomposition}
\begin{array}{ccl}
f_e^+(\theta)&=&\mbox{$\sum_{P\ni e}$}f_{P,e}^+(\theta)\\
f_e^-(\theta)&=&\mbox{$\sum_{P\ni e}$}f_{P,e}^-(\theta)
\end{array}
\end{equation}
such that for all links $e=vw$ and almost all $\theta\in\R$ one has
\begin{equation}\label{pathlink}
f_{P,e}^+(\theta)=\left\{\begin{array}{cl}
h_P(\theta)&\mbox{ if $e$ is the first link of $P$}\\
f_{P,e^*}^-(\theta)& \mbox{ if  $e^*$ is the link in $P$ just before $e$,}
\end{array}\right.
\end{equation}
together with the link transfer equations 
\begin{equation}\label{salida}
\int_{0}^{T_e(\theta)}\!\!\!\! f^-_{P,e}(\xi)\,d\xi=\int_{0}^{\theta}\!\! f^+_{P,e}(\xi)\,d\xi
\end{equation}
where $T_e$ is the link travel time induced by $f_e^+$ through equations \eqref{dam} and \eqref{etf}.
We denote $\omega$ the tuple comprising all the flows $f_e^+,f_{P,e}^+,f_e^-,f_{P,e}^-$ for $e\in E$ and $P\in\P$.
In order to prove the existence and uniqueness of a network loading we
first establish the following technical lemma.
\begin{lemma}\label{proposition:defFmulti}
Let a link-path decomposition of the inflow $f_e^+(\theta)=\sum_{P\ni e}f^+_{P,e}(\theta)$ be given 
over an initial interval $(-\infty,\bar\theta]$. Then there exist unique outflows $f_{P,e}^-\in L^\infty((-\infty,T_e(\bar\theta)])$
satisfying \eqref{salida}, with $0\leq f_{P,e}^-(\xi)\leq \nu_e\,$ for all $\xi\leq T_e(\bar\theta)$.
\end{lemma}
\proof Since $T_e$ maps $(-\infty,\bar\theta]$ surjectively onto $(-\infty,T_e(\bar\theta)]$ it is clear that there 
is at most one $f^-_{P,e}$ satisfying \eqref{salida}.
To establish the existence let $A\subseteq (-\infty,\bar\theta]$ be the set of times $\theta$ at which the 
derivative $T_e'(\theta)$ exists and is strictly positive, and set 
$$f^-_{P,e}(T_e(\theta))=\left\{\begin{array}{cl}
f_{P,e}^+(\theta)/T_e'(\theta)&\mbox{for $\theta\in A$}\\
0&\mbox{otherwise.}
\end{array}\right.$$
This unambiguously defines $f_{P,e}^-(\xi)$ for all $\xi\leq T_e(\bar\theta)$ 
as a non-negative measurable function. Moreover, for $\theta\in A$ we have
$$\sum_{P\ni e}f_{P,e}^-(T_e(\theta))=f_e^+(\theta)/T_e'(\theta)=f_e^-(T_e(\theta))\leq\nu_e$$
which implies $0\leq f_{P,e}^-(\xi)\leq \nu_e$ for all $\xi\leq T_e(\bar\theta)$ so that the $f_{P,e}^-$'s are essentially bounded.
Finally, a change of variables in the integral (see \S\ref{L1AC}) gives
$$\int_0^{T_e(\theta)}\!\!\!f_{P,e}^-(\xi)\,d\xi=\int_0^{\theta}\!\! f_{P,e}^-(T_e(\xi))T_e'(\xi)\,d\xi=\int_0^{\theta}\!\! f_{P,e}^+(\xi)\,d\xi$$
where we used the equality $f_{P,e}^-(T_e(\xi))T_e'(\xi)=f_{P,e}^+(\xi)$ which follows from the definition of $f_{P,e}^-(\xi)$ when 
$\xi\in A$ and from the fact that, almost everywhere, \eqref{der} implies that if $T_e'(\xi)=0$ then $f_e^+(\xi)=0$ and 
therefore $f_{P,e}^+(\xi)=0$.
\endproof

\begin{proposition}\label{proposition:constCaminos}
Suppose that $\tau_e>0$ on all links $e$. Then to each path-flow tuple
 $h$ it corresponds a unique network loading $\omega$.
\end{proposition}
\begin{proof} Let $h=(h_P)_{P\in\P}$ be a given family of path-flows and 
suppose that we have a link-path decomposition satisfying \eqref{decomposition}, \eqref{pathlink}
and \eqref{salida} over an interval $(-\infty,\bar\theta]$. 
For $\bar\theta=0$ this is easy since all flows vanish on the negative axis.
By Lemma \ref{proposition:defFmulti}, the inflow decompositions
$f_e^+(\theta)=\sum_{P\ni e}f_{P,e}^+(\theta)$ over $(-\infty,\bar\theta]$, together with
condition \eqref{salida}, determine unique link-path decompositions for the outflows
$f_e^-(\theta)=\sum_{P\ni e}f_{P,e}^-(\theta)$ over the interval $(-\infty,T_e(\bar\theta)]$. 
These intervals include $(-\infty,\bar\theta+\varepsilon]$ where $\varepsilon=\min_e\tau_e>0$, 
and then using \eqref{pathlink} it follows that the link inflows and their link-path decompositions 
have unique extensions to $(-\infty,\bar\theta+\varepsilon]$. Proceeding 
inductively it follows that the inflows and outflows, together with their link-path decompositions, are 
uniquely defined  on all of $\R$. 
\end{proof}

\subsection{Continuity of path travel times}\label{contimes}
We prove next that the network loading procedure defines path travel time 
maps $h\mapsto \l^P\!$ which are weak-strong continuous from 
$K\subset L^p(0,T)^{\P}$ to the space of continuous functions $C([0,T],\R)$
endowed with the uniform norm. The proof is split into several lemmas.

\begin{lemma}\label{acota}
There exists a constant $M\geq 0$ such that all the flows in the network loading 
corresponding to any $h\in K$ are supported on $[0,M]$.
\end{lemma}
\proof We claim that the queue lengths are bounded by $z_e(\theta)\leq\bar z=\int_0^Tu(\xi)\,d\xi$. 
Indeed, an inductive argument based on \eqref{pathlink} and \eqref{salida} shows that 
for each path $P$ and each link $e\in P$ we have 
$\int_\R f_{P,e}^+(\xi)\,d\xi=\int_\R h_P(\xi)\,d\xi$. Since $z_e(\theta)\leq F_e^+(\theta)$, using 
\eqref{decomposition} we get
$$z_e(\theta)\leq F_e^+(\theta)=\sum_{P\ni e}\int_0^\theta\!\!\! f_{P,e}^+(\xi)\,d\xi\leq\sum_P\int_\R \!h_P(\xi)\,d\xi=\int_0^T\!\!\!u(\xi)\,d\xi.$$
This bound implies that the time to traverse a link $e$ is at most $\bar z/\nu_e+\tau_e$.
Denoting by $\delta$ the maximum of these quantities over all $e\in E$ and setting $M=T+m\delta$ where $m$ 
is the maximum number of links in all paths $P\in\P$ then $\l^P\!(\theta)\leq M$ for all $P\in\P$ and
$\theta\in [0,T]$. This, together with \eqref{pathlink} and \eqref{salida}, implies in turn that all the flows in a 
network loading are supported on the interval $[0,M]$. 
\endproof
 
\begin{lemma}\label{proposition:defFmenos}
The maps $f_e^+\mapsto z_e$ and $f_e^+\mapsto T_e$ defined by \eqref{dam} and \eqref{etf}
are weak-strong continuous 
from $L^p(0,M)$ to $C([0,M],\R)$.
 \end{lemma}

\proof The continuity of $f_e^+\mapsto T_e$ is immediate from that of $f_e^+\mapsto z_e$.
To show the latter we recall that Arzela-Ascoli's theorem implies that the integration map 
$I:L^p(0,M)\to C([0,M],\R)$ defined by $Ix(\theta)=\int_0^\theta x(\xi)d\xi$ is a compact operator, 
and hence it is weak-strong continuous. It follows that the map $f_e^+\mapsto y_e$
given by $y_e(\theta)=\int_0^\theta[f_e^+(\xi)-\nu_e]d\xi$ is weak-strong continuous
and then the same holds for $f_e^+\mapsto z_e$ since \eqref{dam} gives
$$z_e(\theta)=\max_{u\in[0,\theta]}y_e(\theta)\!-\!y_e(u)=y_e(\theta)-\!\!\min_{u\in[0,\theta]}y_e(u)$$
and the map $y\mapsto Hy$ operating on $C([0,M],\R)$ as $Hy(\theta)=\min_{u\in[0,\theta]}y_e(u)$
is nonexpansive.
\endproof

\begin{lemma}\label{omega}
Let $\Omega$ denote the set of all the restrictions to $[0,M]$ of the pairs $(h,\omega)$ where $h\in K$ 
and $\omega$ is the corresponding network loading. Then $\Omega$ is a bounded and weakly closed
subset of $L^p(0,M)^k$ where $k$ is the dimension of the tuple $(h,\omega)$, namely $k=|\P|+2|E|+2|\P||E|$.
 \end{lemma}
 \proof From Lemma \ref{acota} we know that all flows $(h,\omega)\in\Omega$ 
are supported on $[0,M]$, while \eqref{pathlink} and Lemma \ref{proposition:defFmulti} 
imply that they are uniformly bounded in $L^p(0,M)$.
Let us take a weakly convergent net $(h^\alpha,\omega^\alpha)\rightharpoonup (h,\omega)$
with $(h^\alpha,\omega^\alpha)\in\Omega$. It is clear that conditions 
\eqref{decomposition} and \eqref{pathlink} are stable under weak limits
so that $\omega$ satisfies these equations. In order to show \eqref{salida}
it suffices to pass to the limit in 
\begin{equation}\label{salidaalpha}
\int_{0}^{T_e^\alpha(\theta)}\!\!\!\! f^{\alpha-}_{P,e}(\xi)\,d\xi=\int_{0}^{\theta}\!\! f^{\alpha+}_{P,e}(\xi)\,d\xi.
\end{equation}
The right hand side converges to  $\int_{0}^{\theta} f^+_{P,e}(\xi)\,d\xi$
while the integral on the left can be written as the sum 
$$\int_{0}^{T_e^\alpha(\theta)}\!\!\!\! f^{\alpha-}_{P,e}(\xi)\,d\xi=\int_{0}^{T_e(\theta)}\!\!\!\! f^{\alpha-}_{P,e}(\xi)\,d\xi+\int_{T_e(\theta)}^{T_e^\alpha(\theta)}\!\!\!\! f^{\alpha-}_{P,e}(\xi)\,d\xi.$$
The first term on the right converges to $\int_{0}^{T_e(\theta)} f^-_{P,e}(\xi)\,d\xi$ while the second
converges to zero. Indeed, by H\"older's inequality we have
$$\left|\int_{T_e(\theta)}^{T_e^\alpha(\theta)}\!\!\!\! f^{\alpha-}_{P,e}(\xi)\,d\xi\right|\leq \|f^{\alpha-}_{P,e}\|_p\sqrt[q]{|T_e^\alpha(\theta)-T_e(\theta)|}$$
so the conclusion follows since $0\leq f_{P,e}^{\alpha-}(\xi)\leq \nu_e$ implies $\|f_{P,e}^{\alpha^-}\|_p\leq \nu_e\sqrt[p]{M}$ 
while Lemma \ref{proposition:defFmenos} gives $T_e^\alpha(\theta)\to T_e(\theta)$. Hence we may pass to the limit in 
\eqref{salidaalpha} which proves that $w$ satisfies \eqref{salida} and therefore $(h,\omega)\in\Omega$ as was to be proved.
 \endproof

\begin{lemma}
\label{lemma:gammaCont}
The maps $h\mapsto T_e$ defined by the network loading procedure are 
weak-strong continuous from $K\subset L^p(0,T)^{\P}$ to $C([0,M],\R)$.
\end{lemma}
\begin{proof} Take a weakly convergent net $h^\alpha\rightharpoonup h$ in $K$ and
let $\omega^\alpha$ be the corresponding network loading. From Lemma \ref{acota} we know that the 
net $\omega^\alpha$ is bounded in $L^p(0,M)$, while Lemma \ref{omega} implies that 
any weak accumulation point of $w^\alpha$ is a network loading for $h$. Since the latter 
is unique it follows that $w^\alpha\rightharpoonup w$.
In particular $f_e^{\alpha+}\rightharpoonup f_e^+$ weakly in $L^p(0,M)$ 
so that the conclusion $T_e^\alpha\to T_e$ strongly in $C([0,M],\R)$ is a consequence of
Lemma \ref{proposition:defFmenos}.
\end{proof}

\begin{lemma}\label{timecontinuity} For each $P\in\P$ the map $h\mapsto \l^P$ defined by
the network loading procedure is weak-strong continuous from $K\subset L^p(0,T)^{\P}$ to $C([0,T],\R)$.
\end{lemma}
\proof Let $P=(e_1,e_2,\ldots,e_k)$. Set $M_i=T+i\delta$ with $\delta$ as in the proof of Lemma \ref{acota}
and consider the restrictions $T_{e_i}:[0,M_{i-1}]\to[0,M_i]$ so that for all $\theta\in [0,T]$
\begin{align}\label{path_time2}
    \l^P\!(\theta) & = T_{e_k}\circ\cdots\circ T_{e_1}(\theta).
\end{align}
By Lemma \ref{lemma:gammaCont} the maps $h\mapsto T_{e_i}$ are weak-strong continuous,
so the conclusion follows by noting that composition is a continuous operation.
More precisely,  the map $(f,g)\mapsto g\circ f$ defined on the spaces
$$\circ:C([0,M_{i-1}],[0,M_i])\times C([0,M_i],[0,M_{i+1}])\to C([0,M_{i-1}],[0,M_{i+1}]$$
is a continuos map (with respect to uniform convergence). Indeed, consider a strongly 
convergent net $(f^\alpha,g^\alpha)\to(f,g)$. Then for each $\theta\in [0,M_{i-1}]$ we have
$$|g^\alpha\!\circ\! f^\alpha(\theta)-g\!\circ\! f(\theta)|\leq 
|g^\alpha(f^\alpha(\theta))-g(f^\alpha(\theta))|+|g(f^\alpha(\theta))-g(f(\theta))|.$$
The first term on the right can be bounded by $\|g^\alpha-g\|_\infty$ which tends to 0,
while the second term also tends to zero uniformly in $\theta$ since $g$ is uniformly 
continuous and $\|f^\alpha-f\|_\infty$ tends to zero.
\endproof

\subsection{Existence of dynamic equilibrium for a single origin-destination}\label{exi}
With these preliminary results we may now prove that the variational inequality VI$(K,\A)$ 
has a solution, and the corresponding network loading gives a dynamic equilibrium.
\begin{theorem}\label{teo:Lp}
Let $u\in L^p(0,T)$  with $1<p<\infty$ and assume that $\tau_e>0$ on every link $e$. 
Then there exists a dynamic equilibrium.
\end{theorem}

\begin{proof} According to Lemma \ref{timecontinuity} the map $h\mapsto \A(h)$
is weak-strong continuous from $K$ to $X^*$ so that the variational inequality $\VI(K,\A)$
has a solution $h\in K$. We claim that the corresponding flow-over-time $f$ 
given by Proposition  \ref{proposition:constCaminos} is a dynamic equilibrium. 
If not, by Theorem \ref{char} we may find $\theta>0$ and a link $e=vw\not\in E'_\theta$ such that for all $\epsilon>0$ 
we have $f_e(\l_v(\xi))\l_v'(\xi)>0$ on a subset of positive measure in $[\theta,\theta+\epsilon]$. 
Choose $\epsilon$ small enough so that $E'_\xi$ decreases on $[\theta,\theta+\epsilon]$ and 
choose $P\in \P$ with $e\in P$ and $h_P(\xi) >0$ on a subset $I\subseteq[\theta,\theta+\epsilon]$ 
with positive measure. Take $P'\in\P$ with all links in $E'_{\theta+\epsilon}$ so that $P'$ is optimal 
at each $\xi\in[\theta,\theta+\epsilon]$, and let $h' \in K$ be identical to $h$ except for $\xi\in I$ 
where we transfer flow from $P$ to $P'$, that is $h'_P(\xi)=0$ and $h'_{P'}(\xi)=h_{P'}(\xi)+h_P(\xi)$. 
Then we have
$$0\leq \left\langle \A h,h'-h \right\rangle=\int_{[0,T]}\langle\A h(\xi),h'(\xi) - h(\xi)\rangle d\xi
=\int_I(\l_t(\xi)-\l^P\!(\xi))h_P(\xi)d\xi.$$
Since $e\not\in E'_\theta$ it follows that $e\not\in E'_\xi$ for $\xi\in I$ so that
$\l_t(\xi)<\l^P\!(\xi)$ which yields a contradiction.
\end{proof}

\subsection{Extension to multiple origin-destinations}\label{multiOD}
The extension to multiple origin-destinations is straightforward. 
For each pair $st\in N\times N$ let $u_{st}\in L^p(0,T)$ be the corresponding 
inflow (possibly zero) and let $\P_{st}$ be the set of $s$-$t$ paths 
assumed nonempty if $u_{st}$ is nonzero. 
A feasible flow-over-time is now a family of inflows  $f_e^+=\sum_{st}f^+_{e,st}$ and 
outflows $f_e^-=\sum_{st}f^-_{e,st}$ satisfying the capacity and non-deficit constraints,
together with the flow conservation equations for each $st$ pair
\begin{align}\label{eq:flow_consst}
\sum_{e \in \delta^+(v)} f^+_{e,st} (\theta) -\!\!\!\!\sum_{e \in \delta^-(v)} f^-_{e,st} (\theta) 
=\left\{\begin{array}{cl}
u_{st}(\theta)&\quad\mbox{ for $v=s$ }\\[1ex]
0&\quad\mbox{ for $v \in V \setminus \{ s,t \}$}.
\end{array}
\right.
\end{align}
The definitions of queue lengths, link travel times, and path travel times
remain unchanged, and we only need to introduce the origin-destination optimal times
$$\l_{st}(\theta)=\min_{P\in\P_{st}}\l^P\!(\theta).$$
Dynamic equilibrium holds when for each pair $st$ and all $e=vw\in E$ we have
$f_{e,st}^+(\xi)=0$ for almost all $\xi\in \l_{sv}(\R\setminus\Theta^s_e)$ where
 $\Theta_e^s$ denotes the set of all times $\theta$ at which link $e=vw$ is active
 for origin $s$, namely $\l_{sw}(\theta)=T_e(\l_{sv}(\theta))$. 

Denoting $\P$ the union of all the $\P_{st}$'s, the network loading procedure
in \S\ref{netloading} remains unchanged as it did not depend on having a single 
origin-destination pair. Also the results 
in \S\ref{contimes} are easily extended by considering $K$ as the set of 
path-flows 
$h=(h_P)_{P\in\P}\in L^p(0,T)^\P$ which are non-negative and that satisfy flow conservation for
each pair $st$, that is
$$ \sum_{P \in \P_{st}} h_P=u_{st}.$$
For the bound $\bar z=\int_0^Tu(\xi)\,d\xi$ of the queue lengths in Lemma \ref{acota} 
it suffices to take $u$ as the sum of all the $u_{st}$'s. With these preliminaries,
the proof of Theorem \ref{teo:Lp} is readily adapted to establish the existence of a
dynamic equilibrium for multiple origin-destinations.
\begin{theorem}\label{teo:LpOD}
Let $u_{st}\in L^p(0,T)$  with $1<p<\infty$ the inflows for multiple origin-destination pairs
$st\in N\times N$, and assume that $\tau_e>0$ on every link $e$. 
Then there exists a dynamic equilibrium.
\end{theorem}

\section{Concluding remarks}\label{related-work}
Although dynamic traffic assignment has received considerable attention since the seminal work 
by Merchant and Nemhauser \cite{mer,mer1}, the existence and characterization of dynamic equilibria 
still poses many challenging questions. For a review of the literature and open problems we 
refer to Peeta and Ziliaskopoulos \cite{pep}. 
Several of the previous studies have relied on a strict FIFO condition
that requires the exit time functions $T_e(\cdot)$ to be strictly increasing. 
For instance, Friesz {\em et al.} \cite{fri} consider a situation in which 
users choose simultaneously route and departure time, with link travel times 
specified as $D_e(y_e)=\alpha_e y_e+\beta_e$ where $y_e=F_e^+(\theta)-F_e^-(\theta)$ is the 
total flow on link $e$ at time $\theta$ and $\alpha_e, \beta_e$ are strictly positive constants. 
Strict FIFO was shown to hold for such linear volume-delay functions, which allowed to characterize 
the equilibrium by a variational inequality, though no existence result was given. 
Strict FIFO was also used by Xu {\em et al.}  \cite{XuF} to investigate the network loading problem, namely, to 
determine the link volumes and travel times that result from a given set of path flow 
departure rates. Shortly after, the existence of equilibria was established by Zhu and Marcotte \cite{Zhu} 
under a strong FIFO condition that holds for linear volume-delay functions (even in the case
$\alpha_e=0$), assuming in addition that inflows are uniformly bounded.

Unfortunately, as illustrated by the Example in section \S\ref{subsection:eq_def}, strict FIFO does not 
hold in our framework and these previous results do not apply. This is somewhat surprising since we also 
consider linear travel times. The subtle difference is that we consider the queue size $z_e$ instead of 
the total volume $y_e$ on the link. Note that the fluid queueing model
could be cast into the linear volume-delay framework by decomposing each link into a pure 
queueing segment with travel time $z_e/\nu_e$ (that is, $\alpha_e=1/\nu_e$, $\beta_e=0$), 
followed by a link with constant travel time $\tau_e$ (that is, $\alpha_e=0$, $\beta_e=\tau_e$).
Strict FIFO fails precisely because the queueing segment has $\beta_e=0$.
In this respect it is worth noting that our existence results 
do not require strict FIFO, as long as $\tau_e>0$, and Theorem \ref{Teo-existence} holds even
if $\tau_e=0$ on some links.

A general existence result for dynamic network equilibrium beyond strict FIFO 
was recently presented by Meunier and Wagner \cite{meu}. Their model considers 
both route choice and departure time choice, and is based
on a weak form of strict FIFO: the travel time $T_e(\cdot)$ strictly 
increases on any interval on which there is some inflow into the link. 
This weaker property does hold in our context and the result applies provided that 
the inflow $u(\cdot)$ belongs to $L^\infty_{\rm loc}(\R)$.

An interesting feature of the approach in \S\ref{section:EU}, compared with previous existence results,
is that it provides a way to construct the equilibrium. In this respect, our work owes much to
Koch and Skutella \cite{KochSk:2}. There are however several differences. 
On the modeling side, we distinguish the notion of dynamic equilibrium from the stronger dynamic 
equilibrium condition (see Remark 1). Both concepts were used interchangeably in \cite{KochSk:2}, 
although they might differ as shown in the Example in \S\ref{subsection:eq_def}. In particular, our 
Theorem \ref{char} precises \cite[Theorem 1]{KochSk:2} which characterizes dynamic equilibria,
not strong equilibria. 
Also, Theorem \ref{Th:K&S} is an
 extension of  \cite[Theorem 2]{KochSk:2} that applies to the larger class of dynamic 
equilibria and provides a sharper conclusion by including the normalization condition.
The existence and uniqueness for NTF's in Theorems \ref{teo:existTF} and \ref{teo:uniTF} are 
new, and so is the subsequent existence and uniqueness of a dynamic equilibrium  in
Theorems \ref{Teo-existence} and \ref{Teo-uniqueness}.
To the best of our knowledge, the latter uniqueness result has not been observed 
previously in the literature. 

The algorithmic approach in \S\ref{section:EU} raises a number of questions. 
On the one hand, it would be relevant to know if the step sizes computed 
in step (2) of the extension algorithm are bounded away from 0. In this case
the $\theta_k$'s would not accumulate and the equilibrium would be computed in 
finitely many steps for any given horizon $T$. A related question is whether a steady 
state could eventually be attained with $\alpha=\infty$ at some iteration, in which case the
algorithm would be finite.
A weaker but still difficult conjecture is whether the queue sizes $z_e(\cdot)$ remain bounded
as long as the capacity of any $s$-$t$ cut is large enough, for instance larger than the inflow at 
any point in time. The 
difficulty for proving such a claim is that the flow across a cut can be arbitrarily larger than 
the inflow: the queueing processes might introduce delay offsets in such a way that the flow entering the network 
at different times reaches the cut simultaneously at a later date, causing a superposition of flows that
exceeds the capacity of the cut. On the other hand, while it is easy to give a finite algorithm 
to compute thin flows with resetting, the computational complexity of the problem remains 
open. A polynomial time algorithm for this would imply that for piecewise constant inflows 
one could compute a dynamic equilibrium in polynomial time (in input plus output).

Another interesting question is whether the constructive approach in \S\ref{section:EU} can be
adapted to deal with more general inflows $u(\theta)$.  More precisely, let $N(\l,u_0)$ denote the unique labels in 
a normalized thin flow  of value $u_0$ with resetting on the set $E^*$ of all links $e=vw$ with 
$\l_w>\l_v+\tau_e$, and $E'$ the set of links with $\l_w\geq\l_v+\tau_e$ (see Proposition \ref{lemma:rightCont}).
Recalling Theorems \ref{Th:K&S} and
\ref{teo:uniTF},  an equilibrium could be computed by solving the system of ordinary 
differential equations
$$
\begin{array}{l}
\l'(\theta)=N(\l(\theta),u(\theta))\\
\end{array}
$$
with initial condition $\l_v(0)$ equal to the minimum $s$-$v$ travel time with empty queues.
The cumulative flows $x_e(\theta)$ could then be recovered by integrating a measurable selection
of the corresponding thin flows. 
The main difficulty here is that the map $N$ is discontinuous in $\l$ so that
the standard theory and algorithms for ODE's do not apply directly.
A final open problem is to extend the constructive approach to multiple 
origin-destinations.

\section{Appendix: The spaces $\Lp$ and $\AC$}
\label{L1AC}
We denote $\Lp$ the vector space of measurable functions $g:\R\to\R$
such that $|g(\cdot)|^p$ is integrable on every bounded interval. 
Similarly,  $\AC$ is the vector space of functions $h:\R\to\R$ that are 
absolutely continuous on every bounded interval. For a thorough study of absolutely continuous 
functions we refer to \cite[Chapter 3]{leo}. Here we just summarize a few 
facts required in our analysis:
\begin{itemize}
\item For all $1\leq p\leq \infty$ we have $\Lp\subseteq\Ll$.
\item The primitive of any  $g\in\Ll$ belongs to $\AC$.
Conversely, every $h\in\AC$ is differentiable almost everywhere with $h'\in\Ll$ and 
$$h(\theta)=h(0)+\mbox{$\int_0^\theta h'(\xi)d\xi.$}$$
\item If $f,g\in \AC$ then their product $fg$ and minimum $\min\{f,g\}$ are also in $\AC$. 
\item If $f,h\in \AC$ we do not necessarily have $f\circ h\in\AC$, but this holds
if either  $f$ is Lipschitz or $h$ is monotone. In both cases the following chain rule holds for almost all $y\in\R$
$$(f\circ h)'(y)=f'(h(y))h'(y).$$
\item In particular, if $h\in\AC$ is monotone and $g\in\Ll$ we have the change of variable formula
$$\int_{h(a)}^{h(b)}\!\!\!g(\xi)d\xi=\int_a^b \!\!g(h(y))h'(y)dy.$$
\end{itemize}

\vspace{1ex}
\noindent
The following are more specific properties for which we could not
find a reference, so we include a proof.
\begin{lemma}\label{sorgenfrey}
Let $g:\R\to \R_+$  be a nonnegative function in $\Ll$ and $\{(a_i,b_i)\}_{i\in I}$
a possibly uncountable family of intervals. Then $g$ vanishes almost everywhere on each 
$(a_i,b_i)$ iff it vanishes almost everywhere on $\cup_{i\in I}(a_i,b_i)$. The statement also holds
with the latter set replaced by $\cup_{i\in I}[a_i,b_i)$ or $\cup_{i\in I}(a_i,b_i]$.
\end{lemma}
\proof Assume with no loss of generality that all intervals are nonempty.
Since $\mu(A)=\int_Ag(\xi)d\xi$ defines a regular measure on the Borel sets $A\subseteq \R$,
for $\Theta=\cup_{i\in I}(a_i,b_i)$ we have
$$\mu(\Theta)=\sup\{\mu(K):K\mbox{ compact }, K\subseteq\Theta\}.$$
Now, each compact $K\subseteq\Theta$ has a finite subcover $K\subseteq \cup_{k=1}^n (a_{i_k},b_{i_k})$
so that 
$$\mu(K)\leq\sum_{k=1}^n\mu((a_{i_k},b_{i_k}))=\sum_{k=1}^n\int_{a_{i_k}}^{b_{i_k}}\!\!\!g(\xi)d\xi=0.$$
It follows that $\mu(\Theta)=0$ which implies that $g(\xi)=0$ for almost all $\xi\in\Theta$ 
and proves the first statement. 

The other claims follow since all three unions differ in countably many elements. Indeed, consider for instance 
the set $N=\cup_{i\in I}[a_i,b_i)\setminus \cup_{i\in I}(a_i,b_i)$. Each point $z\in N$ must be an endpoint
$z=a_i$ with the corresponding interval 
$(a_i,b_i)$ disjoint from $N$.  It follows that if $a_j\in N$ is another such point,
the corresponding intervals cannot overlap, and therefore there can be at most countably many.
A similar argument shows that  $\cup_{i\in I}(a_i,b_i]\setminus \cup_{i\in I}(a_i,b_i)$
is countable. 
\endproof

\noindent{\sc Remark.}
Lemma \ref{sorgenfrey} does not hold for closed intervals $\cup_{i\in I}[a_i,b_i]$.
In fact,  every function $g$ vanishes almost everywhere on each  
interval $[x,x]$ for $x\in\R$, but not necessarily on $\cup_{x\in\R}[x,x]=\R$.
\vspace{1ex}

\begin{lemma}\label{lemma:sac}
Let $z\in\AC$ with $z(\theta)=0$ for $\theta<0$. Then the following are equivalent
\begin{itemize}
\item[(a)] $z(\theta)\geq 0$ for all $\theta$,
\item[(b)] $z(\theta)\leq 0\Rightarrow z'(\theta)\geq 0$ for almost all $\theta$,
\item[({}c)] $z(\theta)\leq 0\Rightarrow z'(\theta)=0$ for almost all $\theta$.
\end{itemize}

\end{lemma}
\proof Let $N$ be a null set such that $z'(\theta)$ exists for all $\theta\not\in N$.\\[0.5ex]
{\bf [(a)$\Leftrightarrow$({}b)]} Under (a), for all $\theta\not\in N$ with $z(\theta)\leq 0$ we have $z(\theta)=0$
so that $z'(\theta)\geq 0$ which gives (b).
Conversely, suppose that (b) holds but $z(\theta)<0$ for some $\theta$, and consider the
smallest $\theta'$ such that $z(\cdot)$ remains negative on $(\theta',\theta]$. Then $z(\theta')=0$ while (b) implies $z'(\xi)\geq0$ 
for almost all $\xi\in(\theta',\theta)$ from which we get the contradiction
$0>z(\theta)=z(\theta')+\int_{\theta'}^{\theta}z'(\xi)d\xi\geq 0$. \\[0.5ex]
{\bf [({}c)$\Leftrightarrow$(a)]} Clearly ({}c) implies (b) which in turn implies (a).
Conversely, since (a) implies (b), it suffices to show that the 
set $A=\{\theta\not\in N:z(\theta)\leq 0; z'(\theta)>0\}$ is  countable. Indeed, for each $\theta\in A$ we have $z(\theta)=0$
and we may find $\epsilon>0$ such that  $z(\theta')>0$ for all $\theta'\in I_\theta=(\theta,\theta+\epsilon)$. These intervals 
$I_\theta$ do not meet $A$ so they cannot overlap, and therefore there can be at most countably many.
\endproof

\bigskip
\noindent{\bf Acknowledgements.} We thank Martin Skutella for many stimulating discussions on dynamic equilibrium, and to Ra\'ul
Gouet for his valuable assistance with the measure theoretic aspects involved in this study.

\end{document}